\documentclass[a4paper,11pt,twoside,reqno]{amsart}

\usepackage[utf8]{inputenc}
\usepackage[plainpages=false,pdfpagelabels=true]{hyperref}
\usepackage{amssymb,amsthm}
\usepackage[margin=1in]{geometry}
\usepackage{comment}
\usepackage{microtype}
\usepackage[shortlabels]{enumitem}
\usepackage[mathscr]{euscript}
\usepackage{a4wide,amsmath,amssymb,bm,graphicx,hyperref}
\usepackage[english]{babel}
\usepackage{amsthm}
\usepackage[T1]{fontenc}
\usepackage{dsfont}
\usepackage{mathtools}
\usepackage{tabularx}
\usepackage{parskip}

\usepackage{geometry}
\usepackage{verbatimbox}

\usepackage{braket}

\usepackage{marginnote}

\usepackage[dvips]{epsfig}
\usepackage[bf]{caption}

\usepackage{adjustbox}
\usepackage{dirtytalk}
\usepackage{mdframed}

\usepackage{tikz}
\usetikzlibrary{matrix}
\usetikzlibrary{arrows.meta}
\usetikzlibrary{cd}

\newcommand{\subs}[1]{\ensuremath{_{\textrm{#1}}}}

\newcommand{\mtn}{\mathds N}%
\newcommand{\mtz}{\mathds Z}%
\newcommand{\mtr}{\mathds R}%
\newcommand{\mtc}{\mathds C}%
\DeclareMathOperator{\spn}{span}
\DeclarePairedDelimiter\norm{\lVert}{\rVert}
\DeclarePairedDelimiter\abs{\lvert}{\rvert}
\DeclarePairedDelimiter\expect{\langle}{\rangle}

\makeatletter
\let\oldabs\abs
\def\abs{\@ifstar{\oldabs}{\oldabs*}}
\let\oldnorm\norm
\def\norm{\@ifstar{\oldnorm}{\oldnorm*}}

\let\oldexpect\expect
\def\expect{\@ifstar{\oldexpect}{\oldexpect*}}
\makeatother

\newtheorem{lemma}{Lemma}[section]
\newtheorem*{lemma*}{Lemma}
\newtheorem{corollary}[lemma]{Corollary}
\newtheorem{proposition}[lemma]{Proposition}

\newtheorem{theorem}{Theorem}
\newtheorem{thm-with-counter}[lemma]{Theorem}
\theoremstyle{remark}
\newtheorem*{remark}{Remark}
\theoremstyle{definition}
\newtheorem{defn}[lemma]{Definition}
\newtheorem{example}[lemma]{Example}

\newtheorem*{futuretheoreminner}{Theorem~\thefuturetheoreminner}
\newcommand{\thefuturetheoreminner}{} 

\ExplSyntaxOn

\prop_new:N \g_alevel_future_prop

\NewDocumentEnvironment{futuretheorem}{ m o +b }
 {
  \renewcommand{\thefuturetheoreminner}{\ref{#1}}
  \IfNoValueTF{#2}
   {\futuretheoreminner}
   {\futuretheoreminner[#2]}
  #3
 }
 {
  \endfuturetheoreminner
  \prop_gput:Nnn \g_alevel_future_prop { #1 } { #3 }
  \IfValueT{#2}{ \prop_gput:Nnn \g_alevel_future_prop { #1-attr } { #2 } }
 }

\NewDocumentCommand{\pasttheorem}{m}
 {
  \prop_if_in:NnTF \g_alevel_future_prop { #1-attr }
   {
    \begin{theorem}[\prop_item:Nn \g_alevel_future_prop { #1-attr }]
   }
   {
    \begin{theorem}
   }
  \label{#1}
  \prop_item:Nn \g_alevel_future_prop { #1 }
  \end{theorem}
 }

\ExplSyntaxOff
\usepackage{thmtools, thm-restate}

\newcommand{\defeq}{\vcentcolon=}

\theoremstyle{definition}

\newcommand{\s}{\mathbb{S}}

\allowdisplaybreaks[1]

\renewcommand{\epsilon}{\varepsilon}

\numberwithin{equation}{section}

\AtBeginDocument{%
   \def\MR#1{}
}

\title{Global eigenfamilies on closed manifolds}
\author{Oskar Riedler}
\address{Universit\"at M\"unster, Mathematisches Institut\\
Einsteinstr. 62\\
48149 M\" unster\\
Germany}
\email{oskar.riedler@uni-muenster.de}

\author{Anna Siffert}
\address{Universit\"at M\"unster, Mathematisches Institut\\
Einsteinstr. 62\\
48149 M\" unster\\
Germany}
\email{asiffert@uni-muenster.de}

\subjclass[2010]{58E20; 53C43; 58C40}
\keywords{global eigenfamilies; global eigenfunctions; $L^2$-orthogonality}

\begin{document}
\begin{abstract}
We study globally defined $(\lambda,\mu)$-eigenfamilies on closed Riemannian manifolds.
Among others, we provide (non-) existence results for such eigenfamilies, examine topological consequences of the existence of eigenfamilies and classify $(\lambda,\mu)$-eigenfamilies on flat tori.
It is further shown that for $f=f_1+i f_2$ being an eigenfunction decomposed into its real and its imaginary part, the powers $\{f_1^a f_2^b\mid a,b\in\mtn\}$ satisfy highly rigid orthogonality relations in $L^2(M)$. In establishing these orthogonality relations one is led to combinatorial identities involving determinants of products of binomials, which we view as being of independent interest.
\end{abstract}

\maketitle

\section{Introduction}
Let $(M,g)$ be a closed Riemannian manifold and $\lambda,\mu\in\mtc$.
We extend the metric $g$ to a complex bilinear form on the complexified tangent bundle of $M$.
A set $\mathcal F=\{\varphi_i:M\rightarrow\mtc\,\lvert\, i\in I\}$ of complex-valued functions is called
a \textit{$(\lambda,\mu)$-eigenfamily on $M$} if for all $\varphi_i,\varphi_j\in\mathcal F$ we have
\begin{align*}
\Delta\varphi_i=\lambda\,\varphi_i,\quad\quad\kappa(\varphi_i,\varphi_j)=\mu\,\varphi_i\varphi_j.
\end{align*}
Here $\Delta$ denotes the Laplace-Beltrami operator on $M$ and $\kappa$ the conformality operator on $M$, i.e. by definition $\kappa(\varphi,\psi)=g(\nabla \varphi,\nabla \psi)$ for smooth functions $\varphi, \psi:M\rightarrow\mtc$. Functions that belong to some $(\lambda,\mu)$-eigenfamily are called \textit{$(\lambda,\mu)$-eigenfunctions}.

Such eigenfunctions were first discussed by Strichartz in the 1980’s \cite{strichartz-89}, who asked for which Riemannian manifolds they could exist. He believed they could play a role in further extending the methods harmonic analysis to spaces that are not locally symmetric.

More recently eigenfamilies were introduced by Gudmundsson and Sakovich \cite{gudmundsson-sakovich-08} as a tool to construct harmonic morphisms. Other applications include ways to generate $p$-harmonic functions \cite{gudmundsson-sobak-20} and minimal submanifolds of codimension two \cite{gudmundsson-munn-23}.
In the last 15 years, various authors constructed explicit examples for eigenfamilies on different domain manifolds,
e.g. eigenfamilies have been constructed on all the classical Riemannian symmetric spaces \cite{ ghandour-gudmundsson-complex23, ghandour-gudmundsson-explicit23, gudmundsson-sakovich-08, gudmundsson-sakovich-09, gudmundsson-siffert-sobak-22, gudmundsson-sobak-20}.
For more details we refer to the references listed in the fantastic online bibliography \cite{Gudmundsson} which is updated regularly.

So far properties of globally defined eigenfamilies on closed manifolds have not been studied thoroughly.
To our best knowledge there exists only the work \cite{riedler-23} in which the first-mentioned author classified certain eigenfamilies on spheres. The aim of this manuscript is to initiate research on this topic.

Our main results are briefly discussed in what follows.

For globally defined $(\lambda,\mu)$-eigenfamilies on arbitrary closed Riemannian manifolds we prove the following relation between the values $\lambda$ and $\mu$:

\begin{futuretheorem}{thm: lambda-mu}
Let $(M,g)$ be closed and $f:M\to\mtc$ a $(\lambda,\mu)$-eigenfunction. Then $\lambda\leq \mu<0$.
\end{futuretheorem}

An application of work of Uhlenbeck \cite{uhlenbeck-76} shows that $(\lambda,\mu)$-eigenfunctions do not exist on generic metrics. Recall that a property is generic if it holds on a countable intersection of open dense sets.

\begin{futuretheorem}{thm: generic}
Let $(M,g)$ be a closed manifold. The property for a Riemannian metric on $M$ to not admit any non-zero $(\lambda,\mu)$-eigenfunctions $(M,g)\to\mtc$ for any $\lambda,\mu\in\mtc$ is generic in the set of Riemannian metrics of $M$.
\end{futuretheorem}

Indeed, almost all known examples of $(\lambda,\mu)$-eigenfamilies and eigenfunctions are defined on homogeneous spaces. It is then interesting to remark that closed Sasaki manifolds admit a large amount of globally defined eigenfunctions:

\begin{futuretheorem}{thm: sasaki}
Let $(M,g)$ be a closed Riemannian manifold of Sasaki type. Then there is an $N\in\mtn$ and an embedding $(f_1,...,f_N): M \to \mtc^N$
so that every $f_i$ is a $(\lambda_i,\mu_i)$-eigenfunction for appropriate values of $\lambda_i,\mu_i$.
\end{futuretheorem}

In order to prove Theorem\,\ref{thm: sasaki} we show if $\lambda\neq\mu$ and $(M,g)$ is closed, that the $(\lambda,\mu)$-eigenfamilies on $(M,g)$ are equivalent to $(0,0)$-eigenfamilies on a certain cone of $M$, see Corollary\,\ref{cor: cones}. Note that the case $\lambda=\mu$ is not ruled out by Theorem\,\ref{thm: lambda-mu}. In fact this case is realised by the eigenfamilies on flat tori, which can be completely classified:

\begin{futuretheorem}{thm: torus}
Let $\Gamma$ be a lattice in $\mtr^n$ and $M=\mtr^n/\Gamma$.
Denote the dual lattice of $\Gamma$ by $\Gamma^*$, the spectrum of the Laplace-Beltrami operator on $M$ by $\sigma(M)$, and let
$$E_\Delta(\lambda)=\{M\to\mtc, x\mapsto e^{2\pi i\langle k',x\rangle}\mid k'\in \Gamma^*, \|k'\|^2=-\frac{\lambda}{4\pi^2}\}$$
be the standard generators of the eigenspace corresponding to $\lambda\in\sigma(M)$.
\begin{enumerate}
\item Every element of $E_\Delta(\lambda)$ is a $(\lambda,\mu)$-eigenfunction on $M$ with $\mu=\lambda$.
\item If $\mathcal F$ is a $(\lambda,\mu)$-eigenfamily on $M$, then there is a $f\in E_\Delta(\lambda)$ so that $\spn_\mtc(\mathcal F)=\spn_\mtc\{f\}$, in particular $\lambda=\mu$ and $\spn_\mtc(\mathcal F)$ is one-dimensional.
\end{enumerate}
\end{futuretheorem}

We will further see that the existence of eigenfamilies leads to
topological consequences.
The case $\lambda=\mu$ is somewhat singular, since the image degenerates to a circle in $\mtc$. 
Many properties of eigenfamilies of the flat tori carry over to the general case $\lambda=\mu$:

\begin{futuretheorem}{thm: lambda=mu}
Let $(M,g)$ be closed and $f:M\to\mtc$ a $(\lambda,\mu)$-eigenfunction with $\lambda=\mu$ ($f$ not identically zero). Then:
\begin{enumerate}
\item The first Betti-number $\beta_1(M)$ is non-zero.
\item The connected components of the level sets of $f$ are all codimension $1$ minimal closed submanifolds of $M$, and they are all equidistant to one-another.
\end{enumerate}
\end{futuretheorem}

For the general case, i.e. when $\lambda$ does not necessarily 
equal $\mu$, we provide the following Morse theoretic result.

\begin{futuretheorem}{thm: |f|-morse}
Let $(M,g)$ be closed and $f:M\to\mtc$ a $(\lambda,\mu)$-eigenfunction, $f$ not identically zero. Let $r_1<r_2\in[0,\infty]$, then every connected component of $|f|^{-1}(\,(r_1,r_2)\,)$ is either empty or has non-zero first Betti number.
\end{futuretheorem}

In particular $M\setminus f^{-1}(\{0\})$ is not simply connected.

Afterwards we show that, for $f=f_1+i f_2$ being an eigenfunction decomposed into its real and its imaginary part, the powers $\{f_1^a f_2^b\mid a,b\in\mtn\}$ satisfy highly rigid orthogonality relations in $L^2(M)$.
Among others, we establish the following result:
\begin{futuretheorem}{thm: l2-powers}
Let $(M,g)$ be closed and $f=f_1+if_2:M\to\mtc$ a $(\lambda,\mu)$-eigenfunction.
\begin{enumerate}
\item Let $a,b\in\mtn$ not both even. Then:
$$\int_M f_1^af_2^b dV_g= 0,$$
where denotes $dV_g$ the Riemannian volume form of $(M,g)$.
\item Let $a,b\in\mtn$. Then:
$$\int_M f_1^{2a}f_2^{2b}dV_g = \frac{\binom{a+b}{a}}{\binom{2a+2b}{2a}} \int_M f_1^{2a+2b}dV_g.$$
\end{enumerate}
\end{futuretheorem}

In Section\,\ref{sec: l2-family} we apply this to obtain further orthogonality relations for eigenfamilies $\mathcal F$:

\begin{futuretheorem}{thm: l2-family}
Let $(M,g)$ be a closed Riemannian manifold and $\mathcal F=\left\{g_j+ih_j\mid j\in\{1,...,k\}\right\}$ a $(\lambda,\mu)$-eigenfamily on $M$. Let $a_1,...,a_n,b_1,...,b_n\in\mtn$, then:
$$\int_M g_1^{a_1}h_1^{b_1}\cdots g_n^{a_n}h_n^{b_n}\,dV_g=(-1)^{\sum_jb_j}\int_M g_1^{b_1}h_1^{a_1}\cdots g_n^{b_n}h_n^{a_n}\,dV_g.$$
\end{futuretheorem}

\medskip

Finally, we provide some fascinating combinatorial identities. The proof of Theorem\,\ref{thm: l2-powers} requires the non-degeneracy of certain matrices whose coefficients involve products of binomials. We show in Theorem\,\ref{combi} that the determinants of these matrices are (signed) powers of $2$:
\begin{futuretheorem}{combi}
Let $n\in\mtn$, then:
\begin{enumerate}
    \item 
$$\det\left[ \left(\sum_{k=0}^m\binom{\ell}{2k}\binom{2n-\ell}{2(m-k)+1}\right)_{\ell,m \in\{0,...,n-1\}} \right] =(-1)^{n(n-1)/2}\,2^{n(n-1)+1},$$
 \item 
$$\det\left[ \left(\sum_{k=0}^m\binom{\ell}{2k}\binom{2n+1-\ell}{2(m-k)+1}\right)_{\ell,m \in\{0,...,n\}} \right] = (-1)^{n(n+1)/2}\,2^{n^2},$$
\item
$$ \det\left[ \left(\sum_{k=0}^m \binom{\ell}{2k}\binom{2n-\ell}{2(m-k)}\right)_{\ell,m\in\{0,...,n\}} \right] =(-1)^{n(n+1)/2}\,2^{n(n-1)}, $$
\item 
$$\det\left[ \left(\sum_{k=0}^m \binom{\ell}{2k}\binom{2n+1-\ell}{2(m-k)}\right)_{\ell,m\in\{0,...,n\}} \right] =(-1)^{n(n+1)/2}\,2^{n^2}.$$
\end{enumerate}
\end{futuretheorem}

\textbf{Organisation:}
In Section\,\ref{sec-1} we provide preliminaries; we in particular recall the definition of eigenfamilies.
Section\,\ref{sec-2} contains fundamental properties of eigenvalues of eigenfamilies. Section\,\ref{sec: existence} concerns the existence results of Theorems\,\ref{thm: generic},\,\ref{thm: sasaki}.
We classify eigenfamilies on flat tori in Section\,\ref{sec: flat-tori}.
The existence of eigenfamilies leads to
topological consequences, compare Section\,\ref{sec-top}.
We provide $L^2$-orthogonality relations for eigenfunctions and eigenfamilies in Sections\,\ref{sec-4} and\,\ref{sec-5}, respectively.
Finally, in Section\,\ref{sec-det} we show the combinatorial identities stated in Theorem\,\ref{combi}.

\textbf{Acknowledgements:}
Oskar Riedler gratefully acknowledges the support of the Deutsche Forschungsgemeinschaft (DFG, German Research Foundation) - Germany’s Excellence Strategy EXC 2044 390685587, Mathematics Münster: Dynamics-Geometry-Structure. Anna Siﬀert gratefully acknowledges the supports of the Deutsche Forschungsgemeinschaft (DFG, German Research Foundation) - Project-ID 427320536 - SFB 1442.

\section{Preliminaries}
\label{sec-1}
We start by recalling the definition of eigenfamilies, which were introduced by
Gudmundsson and Sakovich in 2008 \cite{gudmundsson-sakovich-08}. 

Throughout, let $(M,g)$ be an $m$-dimensional Riemannian manifold. We denote by $T^{\mtc}M$ the complexification of the tangent bundle $TM$ of $M$ and we extend the metric $g$ to a complex bilinear form on $T^{\mtc}M$. The gradient $\nabla\varphi$ of a complex-valued function $\varphi:(M,g)\to\mtc$ is thus a section of $T^{\mtc}M$.  
The complex linear Laplace-Beltrami operator $\Delta$ on $(M,g)$ is locally given as
$$
\Delta\varphi=\sum_{i,j=1}^m\frac{1}{\sqrt{|g|}} \frac{\partial}{\partial x_j}
\left(g^{ij}\, \sqrt{|g|}\, \frac{\partial \varphi}{\partial x_i}\right).
$$
Furthermore, the complex bilinear conformality operator $\kappa$ is defined by $$\kappa(\varphi,\psi)=g(\nabla \varphi,\nabla \psi),$$ 
where $\varphi,\psi:(M,g)\to\mtc$ are two complex-valued functions.
The notion conformality operator stems from the fact that
$\varphi:(M,g)\rightarrow\mtc$ is weakly horizontally conformal if and only if $\kappa(\varphi,\varphi)=0$.\footnote{Recall that 
$\varphi:(M,g)\rightarrow\mtc$ is called weakly horizontally conformal
if at points $p\in M$ where $d\varphi_p\neq 0$, horizontal angles are preserved by $d\varphi_p$. Weakly horizontally conformal maps are a natural
generalization of Riemannian submersions.}
In local coordinates $\kappa$ reads
$$\kappa(\varphi,\psi)=\sum_{i,j=1}^mg^{ij}\,\frac{\partial\varphi}{\partial x_i}\frac{\partial \psi}{\partial x_j}.$$
We have the following fundamental identity
\begin{equation}\label{equation-basic}
\Delta(\varphi\cdot \psi)=(\Delta\varphi)\,\psi +2\,\kappa(\varphi,\psi)+\varphi\,\Delta\psi,
\end{equation}
which we will use later frequently, in most cases without comment.

\smallskip

We can now recall the definitions of eigenfunctions and eigenfamilies.

\begin{defn}[\cite{gudmundsson-sakovich-08}]
Let $(M,g)$ be a Riemannian manifold.
\begin{enumerate}
    \item A complex valued function $\varphi:M\rightarrow\mtc$
  is called an \textit{eigenfunction on $M$} if it is eigen
both with respect to the Laplace-Beltrami operator $\Delta$ and the conformality
operator $\kappa$, i.e. there exist $\lambda, \mu\in\mtc$ such that
\begin{align*}
\Delta\varphi=\lambda\,\varphi,\quad\quad\kappa(\varphi,\varphi)=\mu\,\varphi^2
\end{align*}
are satisfied.
\item
A set $\mathcal F=\{\varphi_i:M\rightarrow\mtc\,\lvert\, i\in I\}$ of complex-valued functions is called
a \textit{$(\lambda,\mu)$-eigenfamily on $M$} if for all $\varphi_i,\varphi_j\in\mathcal F$ (which are not necessarily distinct) we have
\begin{align*}
\Delta\varphi_i=\lambda\,\varphi_i,\quad\quad\kappa(\varphi_i, \varphi_j)=\mu\,\varphi_i\varphi_j.
\end{align*}
\end{enumerate}
\end{defn}

Clearly, each element of an eigenfamily on $M$ is an eigenfunction on $M$. Further if $\mathcal F=\{\varphi:M\to\mtc\mid i \in I\}$ is a $(\lambda,\mu)$-eigenfamily then for all $d\in\mtn$ the family $\mathcal{F}^d= \{\varphi_{i_1}\cdots \varphi_{i_d}\mid i_1,...,i_d\in I\}$ is again an eigenfamily by (\ref{equation-basic}).

\begin{example}[cf. Example 2.2 in  \cite{gudmundsson-munn-23}]
Consider the unit sphere $\s^{2n-1}\subset\mtc^n=\mtr^{2n}$ and let
$\varphi_j:\s^{2n-1}\rightarrow\mtc$, $j\in\{1,\dots,n\}$ be given by 
\begin{align*}
\varphi_j:(z_1,\dots,z_n)\mapsto z_j.   
\end{align*}
Then $\mathcal F=\{\varphi_1,\dots,\varphi_n\}$ constitutes a $(-(2n-1),-1)$-eigenfamily on $\s^{2n-1}$. More generally any family of homogeneous holomorphic polynomials $\mtc^n\to\mtc$ of the same degree induces an eigenfamiliy on $\s^{2n-1}$.
\end{example}

\begin{example}[cf. Example 3.3.(iii) in \cite{riedler-23}]
For $a,b,c,d\in\mtc$ given consider the maps $\varphi_{a,b,c,d}:\mtc^4\to\mtc$ defined by:
$$\varphi_{a,b,c,d}(z,u,w,v) =  a(z^2w+zu\overline v)+b(zu\overline w - z^2v)+c(u^2\overline v+zuw)+d(u^2\overline w - zuv).$$
Then $\{ \varphi_{a,b,c,d}\lvert_{\s^7}:\s^7\to\mtc \mid a,b,c,d\in\mtc\}$ is a $(-15,-9)$-eigenfamily on $\s^7\subset\mtc^4$. Unless $ad-bc=0$ the polynomial $\varphi_{a,b,c,d}$ is not congruent to any holomorphic polynomial.
\end{example}

\begin{example}[cf. Theorem 4.3 in \cite{ghandour-gudmundsson-explicit23}]
Let $O(n+m)/(O(n)\times O(m))$ be a real Grassmanian, and $A$ a complex symmetric $(n+m)\times (n+m)$ matrix such that $A^2=0$. Denoting by $x_{ij}$ the matrix components of an $x\in O(n+m)$ one has that:
$$\varphi_A: O(n+m)/(O(n)\times O(m))\to \mtc,\qquad  [x]\mapsto \sum_{i,j=1}^{n+m}\sum_{\ell=1}^m A_{ij} x_{i\ell}x_{j\ell}$$
is a well-defined $(-(n+m),-2)$-eigenfunction.
\end{example}

\section{Properties of the eigenvalues
$\lambda$ and $\mu$} 
\label{sec-2}
In this section we provide fundamental properties of the eigenvalues $\lambda$ and $\mu$
of globally defined eigenfamilies $\mathcal F$ on closed Riemannian manifolds $(M,g)$.
Throughout $\sigma(M)$ will denote the spectrum of the Laplace-Beltrami operator on $M$.

\smallskip

Gudmundsson and Sakovich \cite{gudmundsson-sakovich-08} introduced $\lambda$ as a complex number. For globally defined eigenfamilies on closed manifolds, however, we have more information, $\lambda$ is then a negative real number, see e.g. Theorem III.A.I.4 of \cite{berger-spectrum} for a textbook treatment.

\begin{lemma}[e.g. \cite{berger-spectrum}]
Let $M$ be a closed manifold $M$ and $\mathcal F$ a $(\lambda,\mu)$-eigenfamily on $M$. Then $\lambda\in\mtr_{<0}$.  
\end{lemma}

We can now easily establish a necessary condition for the existence of a globally defined
$(\lambda,\mu)$-eigenfamily on $M$. 

\begin{lemma}
\label{lem-el}
Let $(M,g)$ be a closed Riemannian manifold. A necessary condition for the existence of a globally defined
$(\lambda,\mu)$-eigenfamily on $M$ is
$$d^2\mu + d\,(\lambda-\mu)\in\sigma(M)$$
    for all $d\in\mtn$. 
\end{lemma}
\begin{proof}
Let $f$ be an element of a $(\lambda,\mu)$-eigenfamily on $M$.
    Using (\ref{equation-basic}),
    for each $d\in\mtn$, we have
    \begin{align*}
        \Delta f^d=(d^2\mu + d\,(\lambda-\mu))f^d.
    \end{align*}
    Thus $d^2\mu + d\,(\lambda-\mu)\in\sigma(M)$, whence the claim.
\end{proof}

\pasttheorem{thm: lambda-mu}
\begin{proof}
As noted, for all $d\in\mtn$ the function $f^d$ satisfies
$$\Delta f^d=\left(d^2\mu + d\,(\lambda-\mu)\right)\, f^d$$
which implies $\mu\leq 0$. In order to show $\lambda\leq\mu$ we consider the equation:
\begin{equation}
\Delta|f|^2 = 4\kappa(|f|,|f|)+2(\lambda-\mu)|f|^2,\label{eq: delta-|f|^2}
\end{equation}
which will be shown separately, see Lemma\,\ref{lemma: h-identities}. Assuming $\lambda-\mu>0$ then implies that $\Delta|f|^2>0$ everywhere, which is impossible since $M$ is compact and $|f|^2$ is smooth.

The remaining statement, namely that $\mu\neq0$, follows from noting that since $M$ is compact there exist no non-constant weakly horizontally conformal maps $f:M\to\mtc$ (i.e. no maps satisfying $\kappa(f,f)=0$ everywhere). Since we do not have a reference for this fact we provide a proof for convenience of the reader: If $f=f_1+if_2$ were such a map note that $\kappa(f,f)=0$ implies that $df=0$ at any critical point of $f$. Now the boundary of the image $f(M)$ consists entirely of critical values, and then $C\defeq\{x\in \mtr\mid \exists y\in\mtr: x+iy\in\partial f(M)\}$ are critical values of the function $f_1$. However $C$ coincides with the image $f_1(M)$, since if $x\in f_1(M)$ then the line $\{x+it\mid t\in\mtr\}$ must intersect the boundary of $f(M)$. Since $f$ is not the constant map the image $f_1(M)$ contains an open interval, and so the critical values are not of measure $0$, contradicting the Theorem of Sard.
\end{proof}

It is well known that the Laplace-Beltrami operator has a discrete spectrum on compact Riemannian manifolds, see e.g. \cite{Rosenberg}. From this we obtain the following result.

\begin{proposition}
Let $(M,g)$ be a closed Riemannian manifold, then
$$\sigma_{E}(M)\defeq \{ (\lambda,\mu)\in\mtr^2\mid \exists\text{ non-zero $(\lambda,\mu)$-eigenfamily }M\to\mtc\}$$ is a discrete subset of $\mtr_{<0}\times\mtr_{<0}$. 
\end{proposition}
\begin{proof}
This follows immediately from the fact that the Laplace-spectrum $\sigma(M)$ is a discrete subset of $\mtr_{<0}$, and Lemma\,\ref{lem-el}, i.e. if $(\lambda,\mu)\in \sigma_E(M)$ then $d^2\mu + d(\lambda-\mu)\in\sigma(M)$ for all $d\in\mtn$.
\end{proof}

\begin{remark}
The information that:\begin{enumerate}[label=(\roman*)]
\item $\sigma(M)$ is discrete and negative,
\item $(\lambda,\mu)\in\sigma_E(M) \implies d^2\mu+d(\lambda-\mu)\in \sigma(M)$ for all $d\in \mtn$,
\end{enumerate}
is not enough to show that the range of $\mu$-values in $\sigma_E(M)$ is discrete. Indeed, consider the set 
\begin{align}
    \label{set}
\left\{ \left(-d(n+(d-1)(1-\frac1n)), -d^2(1-\frac1n)\right)\,\middle|\, n,d\in\mtn\right\},
\end{align}
which satisfies both of these conditions, but the $\mu$-values are not discrete.
Note that (\ref{set}) is simply a set, its not clear that it arises from the spectrum of the Laplace-Beltrami operator on a compact manifold.
\end{remark}

\section{Existence of eigenfunctions}
\label{sec: existence}
In this section we briefly investigate questions of existence of globally defined eigenfamilies and eigenfunctions.

We begin with the observation that eigenfamilies do not exist at all on generic metrics:

\pasttheorem{thm: generic}

\begin{remark}
We say that a property is generic if it holds on a countable intersection of open dense sets. For Theorem\,\ref{thm: generic} we equip the set of Riemannian metrics with the $C^\infty$-topology.
\end{remark}
\begin{proof}[Proof of Theorem\,\ref{thm: generic}]
By work of Uhlenbeck \cite{uhlenbeck-72, uhlenbeck-76} the property that all eigenspaces of the Laplace operator of a metric $g$ are one-dimensional is generic on the set of Riemannian metrics on $M$. Supposing $f=f_1+if_2:(M,g)\to\mtc$ is a non-zero $(\lambda,\mu)$-eigenfunction, one notes from $\lambda\in\mtr$ that $f_1, f_2$ lie in a common eigenspace of the Laplacian. Since $f_1,f_2$ are linearly independent, e.g. by Theorem\,\ref{thm: l2-family} below, the theorem follows. 
\end{proof}

The Riemannian manifolds for which $(\lambda,\mu)$-eigenfamilies and eigenfunctions are known are rather special, indeed most are homogeneous or even symmetric spaces. In what follows we will show that any closed Sasakian manifold admits a large amount of global eigenfunctions:

\pasttheorem{thm: sasaki}
\begin{remark}
One can show additionally that $\lambda_i = -\mu_i^2 -\sqrt{-\mu_i}(\dim(M)-1)$ and $\kappa(f_i, f_j) = -\sqrt{\mu_i\mu_j}\, f_i f_j$.
\end{remark}

Sasakian manifolds form a rich and much studied class in Riemannian geometry, see e.g. \cite{blair-10, boyer-galicki-08} for an introduction. For the purpose of this article the most important fact is the following characterisation of when a Riemannian manifold admits a compatible Sasaki structure:
\begin{proposition}\label{prop: sasaki-cone}
A Riemannian manifold $(M,g)$ is of Sasaki type if and only if the cone $(\mtr_{>0}\times M, dt^2+t^2g)$ admits a K\"ahler structure for which the complex unit is dilation invariant.
\end{proposition}
\begin{remark}
\begin{enumerate}
\item A Sasaki structure on a manifold $M$ is most often declared by a quadruple $(g,\Phi,\eta,\xi)$, were $g$ is a Riemannian metric, $\Phi$ a CR-structure, $\eta$ a contact form, and $\xi$ the Reeb field of $\eta$. These tensors must be compatible with each other in a suitable sense. A convenient list of relations can be found e.g. in \cite{tanno-69}.
\item If $(g,\Phi,\eta,\xi)$ is a Sasaki structure on $M$ then the complex unit $J:T(\mtr_{>0}\times M)\to T(\mtr_{>0}\times M)$ from Proposition~\ref{prop: sasaki-cone} is given by
\begin{equation}
J=\Phi-t\partial_t\otimes\eta+\frac\xi t\otimes dt.\label{eq: J-cone}
\end{equation}
Here $t$ denotes the coordinate on the $\mtr_{>0}$ factor. The complex unit $J$ is dilation invariant in the sense that $L_{t\partial_t}J=0$, where $L_{t\partial t}$ denotes the Lie derivative by $t\partial_t$.
\end{enumerate}
\end{remark}

The key step to show Theorem\,\ref{thm: sasaki} is a characterisation of $(\lambda,\mu)$-eigenfamilies on $(M,g)$ by $(0,0)$-eigenfamilies on an appropriate cone, similar to Theorem 1 of \cite{riedler-23}. This result, Corollary\,\ref{cor: cones}, is of independent interest. We begin with the following definitions:

\begin{defn}Let $(M,g)$ be a Riemannian manifold of dimension $m\geq 2$.
\begin{enumerate}
\item Let $s>0$. The \emph{cone of slope $1/s$ over $M$} is defined to be:
$$C_s(M) \defeq (\mtr_{>0}\times M, s\, dt^2 +t^2 g).$$
\item We say that a pair $(\lambda,\mu)\in\mtc^2$ is \emph{conical} if $\lambda\neq\mu$ and $\frac{-\mu}{(\mu-\lambda)^2} >0$. In this case define:
\begin{equation}
s(\lambda,\mu) = \frac{-\mu (m-1)}{(\lambda-\mu)^2} , \qquad d(\lambda,\mu) = \frac{\mu(m-1)}{\lambda-\mu}.\label{eq: s,d}
\end{equation}
\end{enumerate}
\end{defn}

\begin{remark}
If $\lambda\neq\mu$ are both real and $-\mu>0$ then $(\lambda,\mu)$ is conical. In particular if $(M,g)$ is closed and $f:M\to\mtc$ is a $(\lambda,\mu)$-eigenfunction with $\lambda\neq\mu$, then $(\lambda,\mu)$ is conical.
\end{remark}

The following is a calculation, which we omit:

\begin{lemma}\label{lemma: delta-kappa-cone}
Let $(M,g)$ be a Riemannian manifold, $s\in\mtr_{>0}$, $f,h:M\to \mtc$ smooth. For $d\in\mtc$ define $f_d:M\times \mtr_{>0}\to\mtc$, $(x,t)\mapsto t^d f(x)$ and $h_d$ similarly. Then:
\begin{align*}
(\Delta_{C_s(M)} f_d)\,(x,t) &= t^{d-2}\left(\Delta_M f(x)+\frac{d(\dim(M)+d-1)}{s}f(x)\right), \\
\kappa_{C_s(M)}( f_d,h_d)\,(x,t) &= t^{2d-2}\left(\kappa_M(f,h)(x) + \frac{d^2}s f(x)h(x)\right).
\end{align*}
\end{lemma}

\begin{corollary}\label{cor: cones}
A family $\mathcal F$ of homogeneous degree $d$ maps from $C_s(M)$ to $\mtc$ is a $(0,0)$-eigenfamily if and only if $\mathcal F\lvert_{\{1\}\times M}$ is a $(-\frac{d(m+d-1)}s,-\frac{d^2}s)$-eigenfamily.
\end{corollary}
\begin{remark}
If $(\lambda,\mu)$ are conical then $(-\frac{d(m+d-1)}s,-\frac{d^2}s)=(\lambda,\mu)$, where $d=d(\lambda,\mu)$, $s=s(\lambda,\mu)$ from equation (\ref{eq: s,d}). In particular for $(\lambda,\mu)$ conical all $(\lambda,\mu)$-eigenfamilies arise as the restriction of a homogeneous $(0,0)$-eigenfamily on an appropriate cone.
\end{remark}
\begin{proof}[Proof of Theorem\,\ref{thm: sasaki}]
We make use of an embedding theorem for Sasaki manifolds due to Ornea and Verbitsky (cf. Theorem 1.2 of \cite{ornea-verbitsky-07}), see also the version of van Coevering (Theorem 3.1 in \cite{vancoevering-11}). The theorem gives a holomorphic embedding  $(f_1,...,f_N):C_1(M)\to \mtc^N$ for some $N$, so that each component $f_i$ is homogeneous of degree $d_i$ for some $d_i$. Here holomorphic means with respect to the K\"ahler structure coming from Proposition\,\ref{prop: sasaki-cone}, i.e. the one given by equation (\ref{eq: J-cone}).

Since every holomorphic function to $\mtc$ from a K\"ahler manifold is automatically a $(0,0)$-eigenfunction, see e.g. Corollary 8.1.9 of \cite{baird-wood-book}, an application of Corollary\,\ref{cor: cones} concludes the proof.\end{proof}
\begin{remark}
If $f_1, f_2$ are holomorphic functions on a K\"ahler manifold $(M,g,J)$ one notes that $\kappa(f_1,f_2)=0$, since the gradients of $f_1, f_2$ take values in the anti-holomorphic subspace of $T^\mtc M$, which is isotropic with respect to $g$. Lemma\,\ref{lemma: delta-kappa-cone} then implies $\kappa(f_i, f_j)= -\sqrt{\mu_i\mu_j}f_if_j$ for the functions $f_1,...,f_N$ constructed in the proof of Theorem\,\ref{thm: sasaki}. In particular if the degrees of the $f_i$ all agree then $\{f_1,...,f_N\}$ is  a $(\lambda,\mu)$-eigenfamily.
\end{remark}

\begin{example}\label{example: sasaki-space-form}
We apply the construction of Theorem~\ref{thm: sasaki} to a family of Sasaki space forms. Let $(g\subs{round},\Phi,\eta,\xi)$ denote the standard Sasaki structure on the unit sphere $S^{2n-1}\subset\mtc^n$. For $\alpha>0$ let
$$g_\alpha\defeq\alpha\,g\subs{round} +(\alpha^2-\alpha) \eta\otimes\eta.$$
Then $(g_\alpha, \Phi,\alpha\eta,\frac1\alpha\xi)$ is a Sasaki structure on $S^{2n-1}$ for which the so-called $\phi$-sectional curvature is constant and equal to $\frac4\alpha-3$, i.e. a Sasaki space form, cf. \cite{blair-10, tanno-69}. Note that the metric $g_\alpha$ is a constant multiple of the metric of a Berger sphere.

With respect to the canonical identification of the cone of $S^{2n-1}$ with $\mtc^n\setminus\{0\}$, the family of maps
$$\left\{f_j:\mtc^n\setminus\{0\}\to\mtc, \quad  (z_1,...,z_n)\mapsto  \|(z_1,...,z_n)\|^{\frac{1}\alpha-1}\cdot z_j\, \mid \,  j\in\{1,...,n\}\,\right\}$$
is holomorphic for the complex structure (\ref{eq: J-cone}) induced by the deformed Sasaki structure $(g_\alpha,\Phi,\alpha\,\eta,\frac1\alpha\xi)$ on $S^{2n-1}$. Since the maps are homogeneous of degree $\frac1\alpha$ the family restricts to a $(-\frac1{\alpha^2}-\frac{2n-2}\alpha,-\frac1{\alpha^2})$-eigenfamily on $(S^{2n-1},g_\alpha)$.
\end{example}

\section{Eigenfamilies on the flat tori}\label{sec: flat-tori}
In this section we classify globally defined eigenfamilies on flat tori.

\smallskip

Let $\Gamma$ be a lattice in $\mtr^n$ and $M=\mtr^n/\Gamma$ a flat torus. 
Recall that the \textit{dual lattice $\Gamma^*$ of $\Gamma$} is equal to $$\Gamma^*=\{ k'\in\mtr^n \mid \langle k,k'\rangle\in\mtz \ \forall k\in\Gamma\}.$$
The Laplace-eigenvalues are parametrized by the dual lattice $\Gamma^*$, i.e.
we have 
\begin{align*}
   \sigma(\Gamma)\defeq\{-4\pi^2 \|k'\|^2\mid k'\in\Gamma^*\},
\end{align*}
see e.g. \cite{berger-spectrum}.
Further, for $\lambda\in\sigma(\Gamma)$ introduce the set $B(\lambda)$ by 
\begin{align*}
B(\lambda):=\{k'\in\Gamma^*\mid -4\pi^2\|k'\|^2=\lambda\}.    
\end{align*}

\begin{proposition}[see e.g. \cite{berger-spectrum}]
We have $\sigma(M)=\sigma(\Gamma)$ and the Laplace-eigenspace corresponding to $\lambda\in\sigma(\Gamma)$ is spanned by
$$E_\Delta(\lambda)=\{M\to\mtc ,\quad x\mapsto \exp(2\pi i \langle k', x\rangle)\mid k'\in B(\lambda)\}.$$
\end{proposition}

The fact that any $(\lambda,\mu)$-eigenfamily is contained in the $\lambda$-eigenspace of the Laplacian allows us to classify globally defined eigenfamilies on flat tori.

\pasttheorem{thm: torus}

\begin{proof}
The first point is immediate, if $f=\exp(2\pi i \langle k',x\rangle)\in E_\Delta(\gamma)$ then:
$$\kappa(f,f)=-4\pi^2 \langle k',k'\rangle \exp(2\pi i \langle k'+k',x\rangle) = \lambda f^2.$$
For the second point, if $\mathcal F$ is a $(\lambda,\mu)$-eigenfamily and $f\in\spn_\mtc(\mathcal F)$, then there are $k'_i\in B(\lambda)$ and $a_i\in\mtc$ so that
$$f=\sum_i a_i \exp(2\pi i\langle k'_i, x\rangle).$$
One necessarily has that $\Delta f^2=2(\mu+\lambda)f^2$, in this case:
\begin{align*}
2(\lambda+\mu)f^2 &=2(\lambda+\mu) \sum_{ij}a_ia_j \exp(2\pi i\langle k_i'+k_j',x\rangle)\\&=\Delta f^2 = -4\pi^2\sum_{ij} a_i a_j \|k_i'+k_j'\|^2 \exp(2\pi i\langle k_i'+k_j',x\rangle).
\end{align*}

If there is a single $a_i\neq0$ in the sum, then there is a coefficient of the form $$-4 a_i^2 \lambda \exp(2\pi i\langle 2k_i',x\rangle)$$ in the sum. In fact since the $k_i$ all have the same norm this is the only way to achieve an $\exp(2\pi i\langle 2k_i',x\rangle)$ factor and it cannot be cancelled. Comparing coefficients then implies that $\mu=\lambda$. Finally $\mu=\lambda$ yields that $\|k_i'+k_j'\|^2=4\|k_i'\|^2$ for all $i,j$. This in turn implies that $k_i'=k_j'$ for all $i,j$, and then $f$ is proportional to an element of $E_\Delta(\lambda)$.

Since all elements of $\spn_\mtc(\mathcal F)$ are proportional to an element of $E_\Delta(\lambda)$ and $E_\Delta(\lambda)$ is a finite set, this implies that $\mathcal F$ is spanned by a single element.
\end{proof}

\section{Topological consequences of the existence of eigenfamilies}
\label{sec-top}

In this section we derive some topological consequences from the existence of a $(\lambda,\mu)$-eigenfunction $f:M\to\mtc$ on a closed manifold $M$. 

The case $\lambda=\mu$ is somewhat singular, in this case the modulus $|f|^2$ becomes constant (cf. Proposition\,\ref{prop-modulus}), and the existence of a $(\lambda,\lambda)$-eigenfunction quickly leads to strong restrictions on the geometry and topology of $M$ (cf. Theorem\,\ref{thm: lambda=mu}).

The key point is to consider the \textit{polar form} $f=|f|e^{ih}$ of a $(\lambda,\mu)$-eigenfunction. One quickly derives strong relations for $\Delta$ and $\kappa$ when  applied to $h$ and $|f|$, whenever the function $h$ is locally defined on some open subset $U\subset M$:

\begin{lemma}\label{lemma: h-identities}Let $(U,g)$ be a Riemannian manifold, not necessarily compact or complete, and let $f:U\to \mtc$ be a $(\lambda,\mu)$-eigenfunction with $\lambda,\mu$ both real and $f(x)\neq0$ for all $x\in U$. Suppose $f(x)=e^{ih(x)}|f(x)|$ for some smooth function $h:U\to\mtr$. Then:
\begin{enumerate}[label=(\roman*)]
\item $\Delta h=0$;
\item $\Delta \ln|f|=\lambda-\mu$;
\item $\kappa(h,|f|)=0$;
\item $\kappa(\ln|f|,\ln|f|)=\kappa(h,h)+\mu$;
\item $\Delta |f|^2 = 4\kappa(|f|,|f|)+2(\lambda-\mu)|f|^2$.
\end{enumerate}
\end{lemma}
\begin{proof}
One has locally:
$$\Delta \ln(f) = \frac{\Delta f}{f}-\frac{\kappa(f,f)}{f^2}=\lambda-\mu.$$
On the other hand $\ln(f) = \ln(|f|) +i h$, and so equations (i) and (ii) follow immediately. Further equation (v) follows from $$\Delta |f|^2 = \Delta e^{2\ln|f|} = 2|f|^2\Delta\ln|f| + 4|f|^2 \kappa(\ln|f|,\ln|f|).$$

Finally note that 
$$\kappa(\ln f, \ln f) = \kappa(\ln|f|,\ln|f|)-\kappa(h,h)+\frac{2i}{|f|}\kappa(|f|, h).$$
Since $\kappa(\ln f,\ln f) = \frac{\kappa(f,f)}{f^2}=\mu$, equations (iii) and (iv) also follow.
\end{proof}

As remarked before, in the case $\lambda=\mu$ the image of an eigenfunction degenerates to a circle in $\mtc$:

\begin{proposition}\label{prop-modulus}
Let $(M,g)$ be closed and connected, let $f:M\to \mtc$ be a $(\lambda,\mu)$-eigenfunction, not identically zero. The following are equivalent:
\begin{enumerate}[label=(\roman*)]
\item $\lambda=\mu$.
\item $|f|^2$ is constant.
\item $f(x)\neq0$ for all $x\in M$. 
\end{enumerate}
\end{proposition}
\begin{proof}
(i)$\implies$(ii) follows from equation (\ref{eq: delta-|f|^2}), since $\lambda=\mu$ implies that $|f|^2$ is subharmonic. For $M$ is compact this is only possible if $|f|^2$ is locally constant.

(ii)$\implies$(iii) is obvious, so we show (iii)$\implies$(i). If $f(x)\neq0$ for all $x\in M$, then $\ln(|f|)$ is well-defined and smooth on $M$. However $\Delta\ln(|f|)=\lambda-\mu$ is constant by Lemma\,\ref{lemma: h-identities}, which is only possible if $\ln(|f|)$ is constant and hence $\lambda=\mu$.
\end{proof}

\begin{remark}
In particular whenever $\lambda\neq\mu$ and $M$ is closed any $(\lambda,\mu)$-eigenfunction $f:M\to\mtc$ must have zeros.
\end{remark}

The flat tori provide examples of manifolds admitting $(\lambda,\mu)$-eigenfunctions with $\lambda=\mu$. In the case of a flat torus the level sets of such an eigenfunction are covered by unions of parallel hyperplanes in $\mtr^n$, in particular they are all totally geodesic. Many properties of this case carry over to the general case:

\pasttheorem{thm: lambda=mu}
\begin{proof}
As noted if $(M,g)$ is closed and $f:M\to\mtc$ is a $(\lambda,\mu)$-eigenfunction with $\lambda=\mu$ then $|f|^2$ is constant, in particular $f$ does not vanish anywhere on $M$.

Let $\theta: M\to S^1$, $x\mapsto \frac{f(x)}{|f(x)|}$. Since $\Delta f=\lambda f$ one finds that $\theta$ is harmonic (as a map with co-domain $S^1$). The existence of such a map implies $\beta_1(M)\neq0$ by standard arguments ($d\theta$ is then a closed $1$-form on $M$, and exactness would imply existence of a non-constant global harmonic function on $M$).

Finally note that locally, for $f=e^{ih}|f|$, the equation
$$0=\kappa(|f|,|f|)=(\kappa(h,h)+\mu)|f|^2$$
implies that the (locally defined) $h$ is a (rescaled) distance function, hence its fibres are equidistant. The condition $\Delta h=0$ then implies that its fibres are also minimal. Since $f(x)=e^{ih(x)}|f|$ where $|f|$ is constant the level sets of $f$ are unions of level sets of $h$, hence point (2) of Theorem\,\ref{thm: lambda=mu} follows.
\end{proof}

As remarked before if $\lambda\neq\mu$ then any $(\lambda,\mu)$-eigenfunction on a closed manifold must admit zeros. We briefly investigate these zeros.

\pasttheorem{thm: |f|-morse}

\begin{proof}
Suppose that $|f|^{-1}(\,(r_1,r_2)\,)$ is non-empty. Let $M_i$ be one of its connected components. If $M_i$ has vanishing first Betti number then there exists a globally defined function $h$ on $M_i$ so that $f=e^{ih}|f|$.

We assume $\lambda\neq\mu$, as else $|f|$ is constant and the result follows from Theorem~\ref{thm: lambda=mu}. By Sard's Theorem the regular values of $|f|$ in $(r_1,r_2)$ have full measure. Furthermore $|f|$ is nowhere locally constant, in particular we can find two distinct regular values $a,b\in |f|(M_i)\cap (r_1,r_2)$. We assume $a<b$.

Denote with $B_i \defeq M_i\cap |f|^{-1}(\{a,b\})$, which is the boundary $M_i\cap |f|^{-1}(\,(a,b)\,)$ and a submanifold of $M_i$. Denoting the normal field of $B_i$ by $N$ one has:
\begin{equation}  
0< \int_{M_i\cap |f|^{-1}(\,(a,b)\,)} \kappa(h,h)\,dV_g = - \int_{M_i\cap |f|^{-1}(\,(a,b)\,)}h \Delta h\,dV_g+ \int_{B_i} h\, g(\nabla h, N)\,dV_g.\label{id-1}  
\end{equation}
Recall that $\Delta h =0$ and that by definition $B_i$ is an open submanifold of $|f|^{-1}(\{a,b\})$. Since $\kappa(h,|f|)=0$ it follows that $\nabla h$ is tangent to $B_i$. Both integrals on the right-hand side of (\ref{id-1}) then vanish, contradicting positivity of $\int_{M_i\cap |f|^{-1}(\,(a,b)\,)}\kappa(h,h)$. The first Betti number of $M_i$ must then be non-zero.
\end{proof}

\begin{remark}
It follows that $\beta_1(M\setminus f^{-1}(\{0\}))\neq0$, in particular $M\setminus f^{-1}(\{0\})$ is not simply connected.
\end{remark}

\begin{corollary}
Let $P:\mtr^m\to\mtc$ be a homogeneous polynomial harmonic morphism (for example $m=2n$ and $P:\mtc^n\to\mtc$ is a holomorphic homogeneous polynomial). Then $\mtr^m\setminus P^{-1}(\{0\})$ is not simply connected.
\end{corollary}
\begin{proof}
Recall from \cite{riedler-23} that $P\lvert_{S^{m-1}}$ is a $(\lambda,\mu)$-eigenfunction for $(\lambda,\mu)=(-(d^2+d(m-2)),-d^2)$, here $d$ is the degree of the polynomial $P$. It follows that $S^{m-1}\setminus (P\lvert_{S^{m-1}})^{-1}(\{0\})$ is not simply connected. Since $P$ is homogeneous this set has the same homotopy type as $\mtr^m\setminus P^{-1}(\{0\})$.
\end{proof}

\section{$L^2$-orthogonality relations for eigenfunctions}
\label{sec-4}
Throughout this section let $M$ be a closed manifold.
Further, let $f=f_1+i\,f_2:M\rightarrow\mtc$ be a $(\lambda,\mu)$-eigenfunction, where $f_1,f_2:M\rightarrow\mtr$ are real-valued functions.
We will show some $L^2$-orthogonality of powers of $f_1,f_2$.
In this context we in particular establish a link of eigenfunctions to combinatorial expressions.

\subsection{$L^2$-orthogonality for $f_1$ and $f_2$}
\label{subsec-41}
This subsection serves as motivation for the later subsections.
By straightforward considerations we provide $L^2$-orthogonality for $f_1$ and $f_2$.

\smallskip

From $\mu\in\mtr$, compare Theorem\,\ref{thm: lambda-mu}, and the fact that $f$ is eigen with respect to the conformality operator $\kappa$, we have by separating real and imaginary parts:
\begin{align}
\label{kappa-1}&\langle\mbox{grad}f_1,\mbox{grad}f_2\rangle=\mu f_1f_2,\\
\label{kappa-2}&\lvert \mbox{grad}f_1\rvert^2-\lvert \mbox{grad}f_2\rvert^2=\mu({f_1}^2-{f_2}^2).
\end{align}
The last equation can be rewritten as
\begin{align*}
  \langle\mbox{grad}(f_1-f_2),\mbox{grad}(f_1+f_2)\rangle=\mu(f_1-f_2)(f_1+f_2).
\end{align*}
Since $\lambda\in\mtr$ we also have $$\Delta f_i=\lambda f_i.$$

Let $(M,g)$ be a closed Riemannian manifold.
Then
\begin{align*}
     \lambda\int_M f^2\,dV_g=\int_M f\Delta f\,dV_g=-\int_M\lvert\mbox{grad}f\rvert^2\,dV_g.
\end{align*}
Hence from (\ref{kappa-2}) we get
\begin{align*}
    \lambda\int_M(f_1^2-f_2^2)\,dV_g=-\mu\int_M (f_1^2-f_2^2)\,dV_g.
\end{align*}
Thus, either $\mu=\lambda=0$, which is not possible since $M$ is compact, or $\int_M (f_1^2-f_2^2)\,dV_g=0$, i.e. the $L^2$-norms of $f_1$ and $f_2$ coincide.

\smallskip

Similarly, from (\ref{kappa-1}) we obtain
\begin{align*}
    \lambda\int_Mf_1f_2\,dV_g=\int_Mf_1\Delta f_2\,dV_g=-\int_M\langle\mbox{grad}f_1,\mbox{grad}f_2\rangle\,dV_g=-\mu\int_Mf_1f_2\,dV_g.
\end{align*}
Thus $\int_Mf_1f_2\,dV_g=0.$ By the preceding considerations we obtain the following lemma.

\begin{lemma}
\label{l2-orth}
Let $(M,g)$ be a closed Riemannian manifold and let $f=f_1+i\,f_2:M\rightarrow\mtc$ be a $(\lambda,\mu)$-eigenfunction, where $f_1,f_2:M\rightarrow\mtr$ are real-valued functions.
Then:
        \begin{align*}
            \int_Mf_1f_2\,dV_g=0\quad\mbox{and}\quad\int_Mf_1^2\,dV_g=\int_M f_2^2\,dV_g,\end{align*}
            i.e. $f_1$ and $f_2$ 
        are $L^2$-orthogonal and have the same $L^2$-norm.
   \end{lemma}

\begin{remark}
Another way to state Lemma~\ref{l2-orth} is to say that a $(\lambda,\mu)$-eigenfunction is an isotropic element of $L^2(M)$ equipped with the complex \emph{bilinear} inner product $(f,g)=\int_M f g\,dV_g$, i.e. $\int_M f^2\,dV_g=0$ for any such function.
\end{remark}

\subsection{$L^2$-orthogonality for powers of $f_1$ and $f_2$}
In this subsection we broadly generalize the result of Lemma\,\ref{l2-orth}, i.e. we prove several results concerning rigid $L^2$-orthogonality relations of powers of $f_1, f_2$.

\smallskip

Throughout let $f=f_1+if_2$ be an eigenfunction on a closed manifold $M$, where $f_1,f_2:M\rightarrow\mtr$. This is equivalent to the existence of a sequence of numbers $\lambda_n\in\mtr_{<0}$ so that $\Delta (f_1+if_2)^n = \lambda_n (f_1+if_2)^n$ for all $n$.

\pasttheorem{thm: l2-powers}

\begin{example} Let $n\geq1$, the map $\mtr^{n+1}\to\mtc$, $(x_1,...,x_{n+1})\mapsto x_1+ix_2$ restricts to a $(-n,-1)$-eigenfunction on $S^n$, and one has \cite{folland-01}
$$\int_{S^n}x_1^{2a} x_2^{2b} \,dV_g=\frac{(2a)!(2b)!}{2^{2a+2b}a!b!}\frac{2\pi^{n/2}}{\Gamma(a+b+\frac n2)}=\frac{\binom{a+b}a}{\binom{2a+2b}{2a}}\frac{2\pi^{n/2}\,\Gamma(a+b+1)}{2^{2a+2b}\,\Gamma(a+b+\frac n2)},$$
where $\Gamma$ denotes the gamma function.
\end{example}

In order to prove Theorem\,\ref{thm: l2-powers} we make use of the fact that all powers $f^d$ are eigenfunctions of the Laplacian corresponding to different eigenvalues, and hence orthogonal to each other with respect to the $L^2$ scalar product:

\begin{lemma}\label{lemma: l2-re-im}
For all $a,b\in\mtn$:
$$\int_M \mathrm{Re}(f^a)\mathrm{Im}(f^b)\,dV_g= 0,$$
further if $a\neq b$ then:
$$\int_M \mathrm{Re}(f^a)\mathrm{Re}(f^b)\,dV_g= \int_M \mathrm{Im}(f^a)\mathrm{Im}(f^b)\,dV_g= 0.$$
\end{lemma}
\begin{proof}
In the case $a=b$ the equation $\int_M \mathrm{Re}(f^a)\mathrm{Im}(f^a) = 0$ follows since the real and imaginary parts of a $(\lambda,\mu)$-eigenfunction are $L^2$-orthogonal by Lemma\,\ref{l2-orth}. All other cases follow from the fact that Laplace-eigenfunctions of different eigenvalues are $L^2$-orthogonal to each other.
\end{proof}

When considering the expressions
\begin{align*}
\int_M\mathrm{Re}(f^a)\mathrm{Re}(f^b)\,dV_g&= \sum_{m=0}^{\lfloor(a+b)/2\rfloor}\sum_{k=0}^m \binom{a}{2k}\binom{b}{2(m-k)} (-1)^m\int_M f_1^{(a+b)-2m}f_2^{2m}\,dV_g,\\
\int_M\mathrm{Re}(f^a)\mathrm{Im}(f^b)\,dV_g&= \sum_{m=0}^{\lfloor(a+b-1)/2\rfloor}\sum_{k=0}^m \binom{a}{2k}\binom{b}{2(m-k)+1} (-1)^m\int_M f_1^{(a+b)-2m-1}f_2^{2m+1}\,dV_g.
\end{align*}
the following matrices arise naturally:
\begin{defn}
For $n\in\mtn$ and $0\leq \ell,m\leq \lfloor (n-1)/2\rfloor$ let $$A(n)_{\ell, m}\defeq \sum_{k=0}^m\binom{\ell}{2k}\binom{n-\ell}{2(m-k)+1}.$$
Similarly for $0\leq \ell\leq (n-1)$, $0\leq m\leq n$ define:
$$B(2n)_{\ell, m}\defeq\sum_{k=0}^m\binom{\ell}{2k}\binom{2n-\ell}{2(m-k)}.$$
\end{defn}

In order to show Theorem\,\ref{thm: l2-powers} we suppose the following combinatorial facts, which we will show separately (see Subsection\,\ref{subsec: A-B-ranks}):

\begin{lemma}\label{lemma: A-B-ranks}
Let $n\in\mtn$. Then the $\lfloor(n+1)/2\rfloor\times \lfloor(n+1)/2\rfloor$ matrix $A(n)$ is invertible and the $n\times (n+1)$ matrix $B(2n)$ has a one-dimensional kernel spanned by the vector with components:
$$v_m \defeq (-1)^m\frac{\binom{n}{m}}{\binom{2n}{2m}},\quad m\in\{0,...,n\}.$$
\end{lemma}

\begin{proof}[Proof of Theorem\,\ref{thm: l2-powers}]
Let $n\in\mtn$, then by Lemma\,\ref{lemma: l2-re-im} one has for all $\ell\in\{0,..., \lfloor(n-1)/2\rfloor\}$:
$$0=\int_M \mathrm{Re}(f^{\ell})\mathrm{Im}(f^{n-\ell})\,dV_g= \sum_{m=0}^{\lfloor (n-1)/2\rfloor}A(n)_{\ell, m}\, (-1)^m\int_M f_1^{n-(2m+1)}f_2^{2m+1}\,dV_g.$$
Since $A(n)$ is invertible this implies $\int_M f_1^{n-(2m+1)}f_2^{2m+1}\,dV_g=0$ for all $n,m$. Part (1) of Theorem\,\ref{thm: l2-powers} then follows if $b$ is odd. The case of $a$ being odd then follows from applying this result to the eigenfunction $if = -f_2+if_1$.

\smallskip

For part (2) let $n\in\mtn$, then for all $\ell\in\{0,..., n-1\}$:
$$0=\int_M \mathrm{Re}(f^{\ell})\mathrm{Re}(f^{2n-\ell})\,dV_g=\sum_{m=0}^{n}B(2n)_{\ell, m}\, (-1)^m\int_M f_1^{2n-2m}f_2^{2m}\,dV_g.$$
Since $B(2n)$ has a one-dimensional kernel it follows that there must be some constant $C(n,f)$ so that
$$\int_{M}f_1^{2n-2m}f_2^{2m}\,dV_g = C \frac{\binom{n}{m}}{\binom{2n}{2m}}$$
for all $m\in\{0,...,m\}$, from which part (2) of Theorem\,\ref{thm: l2-powers} follows.
\end{proof}

\begin{remark}
The invertibility of $A(n)$ 
and surjectivity of $B(2n)$,
as expressed by Lemma\,\ref{lemma: A-B-ranks}, can also be seen from the fascinating combinatorial identities:
\begin{gather}
\det\left[ \left(\sum_{k=0}^m\binom{\ell}{2k}\binom{2n-\ell}{2(m-k)+1}\right)_{\ell,m \in\{0,...,n-1\}} \right] =(-1)^{n(n-1)/2}2^{n(n-1)+1}, \label{eq: det-A1}\\
\det\left[ \left(\sum_{k=0}^m\binom{\ell}{2k}\binom{2n+1-\ell}{2(m-k)+1}\right)_{\ell,m \in\{0,...,n\}} \right] = (-1)^{n(n+1)/2}2^{n^2}, \label{eq: det-A2}\\
\det\left[ \left(\sum_{k=0}^m \binom{\ell}{2k}\binom{2n-\ell}{2(m-k)}\right)_{\ell,m\in\{0,...,n\}} \right] =(-1)^{n(n+1)/2} 2^{n(n-1)},\label{eq: det-B1}\\
\det\left[ \left(\sum_{k=0}^m \binom{\ell}{2k}\binom{2n+1-\ell}{2(m-k)}\right)_{\ell,m\in\{0,...,n\}} \right] =(-1)^{n(n+1)/2} 2^{n^2}.\label{eq: det-B2}
\end{gather}
To the best of our knowledge these identities are new. We provide a proof of equations (\ref{eq: det-A1}) - (\ref{eq: det-B2}) in Section\,\ref{sec-det}. 
\end{remark}

\subsection{Proof of Lemma\,\ref{lemma: A-B-ranks}}\label{subsec: A-B-ranks}

In this section we show Lemma\,\ref{lemma: A-B-ranks}. We first show that $A$ and $B$ are surjective, begin by defining for $n\in\mtn$ and $\ell$ either in $\{0,...,\lfloor \frac{n-1}2\rfloor\}$ or $\{0,...,n-1\}$:
$$\alpha_\ell(t)=4\sum_{m=0}^{\lfloor(n-1)/2\rfloor} A(n)_{\ell,m}\, t^{2m+1},\qquad \beta_\ell(t) =4\sum_{m=0}^{n} B(2n)_{\ell,m}\, t^{2m}.$$
A simple calculation shows that:
\begin{align}
\alpha_\ell(t)&=(1+t)^n-(1-t)^n - (1+t)^\ell (1-t)^{n-\ell}  +(1+t)^{n-\ell}(1-t)^\ell,\label{eq: alpha(t)}\\
\beta_\ell(t) &=(1+t)^{2n}+(1-t)^{2n}+(1+t)^{\ell}(1-t)^{2n-a}+(1+t)^{2n-\ell}(1-t)^{\ell}.\label{eq: beta(t)}
\end{align}
And then:
\begin{lemma}\label{lemma: alpha-beta}
Let $k,\ell\in \mtn$ and $\ell\in\{0,...,\lfloor \frac{n-1}2\rfloor\}$ resp. $\{0,...,n-1\}$.
\begin{enumerate}
\item If $k< \ell$ then:
$$\frac{d^k}{dt^k}\alpha_\ell(1) = \frac{n!}{(n-k)!}2^{n-k},\qquad \frac{d^{k}}{dt^{k}}\beta_\ell(1) = \frac{2n!}{(2n-k)!}2^{2n-k}.$$
\item If $k=\ell$ then:
$$\frac{d^k}{dt^k}\alpha_\ell(1) = \left(\frac{n!}{(n-k)!}+(-1)^k k!\right) 2^{n-k},\quad \frac{d^{k}}{dt^{k}}\beta_\ell(1) = \left(\frac{(2n)!}{(2n-k)!}+(-1)^kk!\right) 2^{2n-k}.$$
\item If $\ell=0$ then:
$$\frac{d^k}{dt^k}\alpha_0(1) =\frac{n!}{(n-k)!}2^{n-k+1},\quad \frac{d^{k}}{dt^{k}}\beta_0(1) = \frac{2n!}{(2n-k)!}2^{2n-k+1}.$$
\end{enumerate}
\end{lemma}
\begin{proof}
Follows from equations (\ref{eq: alpha(t)}) and (\ref{eq: beta(t)}).
\end{proof}

\begin{lemma}\label{lemma: A-B-surj}
For all $n\in\mtn$ the matrices $A(n)$ and $B(2n)$ are surjective.
\end{lemma}
\begin{proof}
We only show this for $A(n)$, the proof for $B(2n)$ is almost identical. Let $e_0,...,e_{\lfloor(n-1)/2\rfloor}$ denote the standard basis of $\mtr^{\lfloor(n-1)/2\rfloor}$, we will show that all $e_\ell$ lie in the image of $A(n)$.

For $0\leq k\leq\lfloor(n-1)/2\rfloor$ define the vectors
$$X^k\defeq\frac{(n-k)!}{n!2^{n-k}}\sum_\ell\frac{d^k}{dt^k}\alpha_\ell(1) e_\ell,$$
which lie in the range of $A(n)$.

Lemma\,\ref{lemma: alpha-beta} shows that the $e_\ell$ components of $X^k$ with $\ell>k$ are just $1$, that the $e_k$ component is $1+\frac{(-1)^k}{\binom nk}$, and that the $e_0$ component is $2$, i.e. $X^k$ is of the form:
$$X^k = 2 e_0 + \sum_{\ell=1}^{k-1} (X^k)_\ell \,e_\ell + \left(1+\frac{(-1)^k}{\binom{n}{k}}\right)e_k+\sum_{\ell>k} e_\ell$$
Then:
$$X^k-X^{k+1} = \sum_{\ell=1}^{k}(X^k-X^{k+1})_\ell\,e_\ell+\frac{(-1)^{k}}{\binom{n}{k+1}} e_{k+1}$$
and we conclude by induction that $e_\ell$ is in the range of $A(n)$ for $\ell\in\{1,...,\lfloor\frac{n-1}2\rfloor\}$. Since one has:
$$e_0= \frac12 X^0-\frac12\sum_{\ell\geq 1}e_\ell$$
the remaining basis element $e_0$ must also be in the image of $A(n)$, finishing the proof that $A(n)$ is surjective.
\end{proof}

From dimensional reasons it follows from Lemma\,\ref{lemma: A-B-surj} that $A(n)$ is then invertible and that $B(2n)$ has a $1$-dimensional kernel. The final step in showing Lemma\,\ref{lemma: A-B-ranks} is then the following statement:

\begin{lemma}
For all $n\in\mtn$ and $\ell\in\{0,...,n-1\}$ one has:
$$\sum_{m=0}^nB(n)_{\ell, m}\,(-1)^m \frac{\binom{n}{m}}{\binom{2n}{2m}}=0.$$
\end{lemma}
\begin{proof}
This amounts to verifying the following identity for all $\ell\in\{0,...,n-1\}$:
$$\sum_{m=0}^n\sum_{k=0}^m(-1)^m \frac{\binom{\ell}{2k}\binom{2n-\ell}{2(m-k)}\binom{n}{m}}{\binom{2n}{2m}}=0.$$
By expanding the binomial coefficients as quotients over factorials one notes that the above sum vanishes if and only if
$$\sum_{m=0}^n(-1)^m\binom{n}{m}\sum_{k=0}^m \binom{2m}{2k}\binom{2n-2m}{\ell-2k}=0$$
for all $\ell\in\{0,...,n-1\}$. We now consider the expression:
\begin{align*}
\sum_{\ell=0}^{2n}\sum_{k=0}^m \binom{2m}{2k}\binom{2n-2m}{\ell-2k}t^\ell &= \frac12\left[(1+t)^{2m}+(1-t)^{2m}\right] (1+t)^{2n-2m}\\
&= \frac12(1+t)^{2m}(1+t)^{2(n-m)}+\frac12(1-t)^{2m}(1+t)^{2(n-2m)}.
\end{align*}
Which gives:
\begin{align*}
\sum_{\ell=0}^{2n}&\sum_{m=0}^n(-1)^m\binom{n}{m}\sum_{k=0}^m \binom{2m}{2k}\binom{2n-2m}{\ell-2k}t^\ell=\\
&\frac12 \left[(1+t)^{2}-(1+t)^{2}\right]^n+\frac12\left[(1+t)^2-(1-t)^2\right]^{n}=2^{2n-1}t^n.
\end{align*}
In particular the $t^\ell$ coefficient vanishes for all $\ell\in\{0,...,n-1\}$, which implies the desired identity.
\end{proof}

\section{$L^2$-orthogonality relations for eigenfamilies}\label{sec: l2-family}
\label{sec-5}
In this section we apply Theorem\,\ref{thm: l2-powers} to an eigenfamily $\mathcal F$. One obtains orthogonality relations similar to point (1) of Theorem\,\ref{thm: l2-powers}, see Corollary\,\ref{cor: orth} below.

\smallskip

\pasttheorem{thm: l2-family}
\begin{proof}
Since $\mathcal F$ is a $(\lambda,\mu)$-eigenfamily one has that
$$f\defeq\sum_j e^{-i\varphi_j}(g_j+ih_j) = \sum_j (\cos(\varphi_j) g_j+\sin(\varphi_j)h_j )+ i\sum_j (\cos(\varphi_j)h_j-\sin(\varphi_j)g_j)$$
is a $(\lambda,\mu)$-eigenfunction for all $\varphi_j\in\mtr$. Recall from Theorem\,\ref{thm: l2-powers} that $\int_M \mathrm{Re}(f)^a\,dV_g = \int_M \mathrm{Im}(f)^a\,dV_g$ for all $a\in \mtn$. Expanding yields
\begin{align*}
\int_M\mathrm{Re}(f)^a\,dV_g&= \sum_{\sum_j (\ell_j+ k_j)=a} \binom{a}{\ell_1, ... , \ell_m, k_1,...,k_m} \left(\prod_j\cos(\varphi_j)^{\ell_j}\sin(\varphi_j)^{k_j}\right) \int_M \prod_j g_j^{\ell_j}h_j^{k_j}\,dV_g,\\
\int_M\mathrm{Im}(f)^a\,dV_g&= \sum_{\sum_j (\ell_j+ k_j)=a} \binom{a}{\ell_1, ... , \ell_m, k_1,...,k_m} \left(\prod_j\cos(\varphi_j)^{\ell_j}\sin(\varphi_j)^{k_j}\right)(-1)^{\sum_jk_j} \int_M \prod_j g_j^{k_j}h_j^{\ell_j}\,dV_g.
\end{align*}
and comparing the orders of $\cos(\varphi_j)$, $\sin(\varphi_j)$ then implies the result.
\end{proof}
Applying Theorem\,\ref{thm: l2-family} twice immediately implies the following corollary:
\begin{corollary}[Orthogonality]\label{cor: orth}
Let $(M,g)$ be a closed Riemannian manifold and $\mathcal F=\left\{g_j+ih_j\mid j\in\{1,...,n\}\right\}$ a $(\lambda,\mu)$-eigenfamily on $M$. If $a_1,...,a_n,b_1,...,b_n\in\mtn$ are so that $\sum_j (a_j+b_j)$ is odd then:
$$\int_M g_1^{a_1}h_1^{b_1}\cdots g_n^{a_n}h_n^{b_n}\,dV_g=0.$$
\end{corollary}
\begin{remark}
In point (1) of Theorem\,\ref{thm: l2-powers} one requires only that one of $a,b$ be odd, in particular the sum $a+b$ can be even. For Corollary\,\ref{cor: orth} it is however necessary that the sum $\sum_j(a_j+b_j)$ is odd, since if $f=f_1+if_2$ is a $(\lambda,\mu)$-eigenfunction one has that $\{f_1+if_2, -f_2+if_1\}$ is a $(\lambda,\mu)$-eigenfamily and
$$\int_M f_1^{a_1+b_2}f_2^{b_1+a_2}\,dV_g=(-1)^{a_2}\int_M f_1^{a_1}f_2^{b_1}(-f_2)^{a_2}f_1^{b_2}\,dV_g,$$
which can be non-zero even if both of $a_1+a_2$, $b_1+b_2$ are odd.
\end{remark}

\section{Combinatorial Identities}
\label{sec-det}
In this section we show the combinatorial identities (\ref{eq: det-A1})-(\ref{eq: det-B2}), i.e. the goal of this section is to prove the following theorem:
\pasttheorem{combi}

\begin{example}
We provide one example for each part of Theorem\,\ref{combi}.
\begin{enumerate}
\item Example for part (1) of Theorem\,\ref{combi}: For $n=4$ we have 
\begin{align*}
\det \begin{pmatrix}
8 & 56 & 56&8  \\
7 & 35 & 21&1 \\
6 & 26 & 26&6 \\
5 & 25 & 31&3
\end{pmatrix}=2^{13}.
\end{align*}
\item Example for part (2) of Theorem\,\ref{combi}: For $n=3$ we have 
\begin{align*}
\det\begin{pmatrix}
7 & 35 & 21&1  \\
6 & 20 & 6&0 \\
5 & 15 & 11&1 \\
4 & 16 & 12&0
\end{pmatrix}=2^{9}.
\end{align*}
\item Example for part (3) of Theorem\,\ref{combi}: For $n=5$ we have 
\begin{align*}
\det \begin{pmatrix}
1 & 1 & 1&1&1&1  \\
45 & 36 & 29&24&21&20  \\
210 & 126 & 98&98&106&110  \\
210 & 84 & 98&112&106&100  \\
45 & 9 & 29&21&21&25  \\
1 & 0 & 1&0&1&0  
\end{pmatrix}=-2^{20}.
\end{align*}
\item Example for part (4) of Theorem\,\ref{combi}: For $n=4$ we have 
$$\det 
\begin{pmatrix}
 1 & 36 & 126 & 84 & 9 \\
 1 & 28 & 70 & 28 & 1 \\
 1 & 22 & 56 & 42 & 7 \\
 1 & 18 & 60 & 46 & 3 \\
 1 & 16 & 66 & 40 & 5 
\end{pmatrix} = 2^{16}.$$
\end{enumerate}
\end{example}

\begin{defn}
Let $n\in \mtn$, $\ell\in\{0,...,n\}$. Define $a_\ell(n), b_\ell(n)$ to be the row vectors with components:
\begin{align*}
a_\ell(n)_m &=  \sum_{k=0}^m \binom{\ell}{2k}\binom{n-\ell}{2(m-k)+1},  & &m\in\{0,...,\lfloor \frac{n-1}2\rfloor\},\\
b_\ell(n)_m &=  \sum_{k=0}^m \binom{\ell}{2k}\binom{n-\ell}{2(m-k)},  & &m\in\{0,...,\lfloor \frac{n}2\rfloor\}. 
\end{align*}
For this section we also define:
$$A(n) = \begin{pmatrix} a_0(n)\\ \vdots \\ a_{ \lfloor \frac{n-1}2\rfloor}\end{pmatrix}, \qquad B(n) = \begin{pmatrix}b_0(n)\\ \vdots \\ b_{\lfloor \frac n2\rfloor}(n)\end{pmatrix}.$$
\end{defn}
\begin{remark}\begin{enumerate}
\item The above definition of $A(n)$ agrees with the definition in Section\,\ref{sec: l2-family}, for $B(n)$ we have removed the last row in order for the matrix to be invertible.
\item Theorem\,\ref{combi} can be restated as:
\begin{align*}
\det A(2n) &= (-1)^{n(n-1)/2}2^{n(n-1)+1},& &A(2n+1) = (-1)^{n(n+1)/2}2^{n^2},\\
\det B(2n) &= (-1)^{n(n+1)/2}2^{n(n-1)}, & & B(2n+1) = (-1)^{n(n+1)/2}2^{n^2}.
\end{align*}
\item $a_{\ell}(n)$ is a row vector of length $\lfloor(n+1)/2\rfloor$. At times it is however convenient to increase its length by adding a tail of zeros, e.g. if $n$ is even then for the equation $a_1(n+1)=a_0(n)$ (proven below) to hold one must increase the length of $a_0(n)$ in this way.
\end{enumerate}\end{remark}

The key steps to showing the above identities are the following relations:

\begin{proposition}\label{prop: r-relations}
Let $n\in\mtn$, $\ell\in\{0,...,n\}$. Then:
\begin{enumerate}
\item $\displaystyle \qquad a_0(n)=a_\ell(n)+a_{n-\ell}(n).$
\item $\displaystyle\qquad a_{\ell+1}(n+1)=2(-1)^\ell \sum_{j=0}^\ell (-1)^j a_j(n) +\begin{cases}-a_0(n) & \ell\text{ even}\\ a_0(n+1) & \ell\text{ odd}\end{cases}.$
\end{enumerate}
\end{proposition}
\begin{proof}
Let $P_{\ell,n}(x)$ denote the polynomial
$$4\sum_m \sum_{k=0}^m \binom{\ell}{2k}\binom{n-\ell}{2(m-k)+1} x^{2m+1},$$
i.e. the $x^{2m+1}$ component of $P_{\ell,n}$ is $4$ times $a_{\ell}(n)_m$. One checks:
\begin{align*}
P_{\ell,n}(x) &= \left((1+x)^\ell +(1-x)^\ell\right)\cdot \left((1+x)^{n-\ell}- (1-x)^{n-\ell}\right)\\
&= (1+x)^n-(1-x)^n +(1+x)^{n-\ell}(1-x)^\ell - (1+x)^\ell(1-x)^{n-\ell}.
\end{align*}
Further, one immediately sees:
\begin{align*}
P_{0,n} = 2(1+x)^n-2(1-x)^n = P_{\ell,n}+P_{n-\ell,n},
\end{align*}
which implies part (1).

For part (2) we first calculate:
$$\sum_{j=0}^\ell (-1)^j(1+x)^j(1-x)^{n-j}=(1-x)^n \frac{1-(\frac{1+x}{x-1})^{\ell+1}}{1-(\frac{1+x}{x-1})}=\frac{(1-x)^{n+1}}{2}-(-1)^\ell\frac{(1-x)^{n-\ell}(1+x)^{\ell+1}}2.$$
Similarly:
$$\sum_{j=0}^\ell (-1)^j(1-x)^j(1+x)^{n-j}=(1+x)^n \frac{1-(\frac{x-1}{1+x})^{\ell+1}}{1-(\frac{x-1}{1+x})}=\frac{(1+x)^{n+1}}{2}+(-1)^\ell\frac{(1+x)^{n-\ell}(1-x)^{\ell+1}}2.$$
Additionally
$$\sum_{j=0}^\ell (-1)^j (1+x)^n = \frac{1+(-1)^\ell}2(1+x)^n, \quad \sum_{j=0}^\ell (-1)^j (1-x)^n = \frac{1+(-1)^\ell}2(1-x)^n.$$
Together this implies:
\begin{align*}
2\sum_{j=0}^\ell (-1)^jP_{j,n}(x) =(-1)^{\ell} &\left((1+x)^{n-\ell}(1-x)^{\ell+1}-(1+x)^{\ell+1}(1-x)^{n-\ell}\right)\\
&+ (1+x)^{n+1}-(1-x)^{n+1}+ (1+(-1)^\ell)\left((1+x)^n-(1-x)^n\right).
\end{align*}
If $\ell$ is even this equation becomes:
$$2\sum_{j=0}^\ell(-1)^j P_{j,n}(x) = P_{\ell+1,n+1}(x)+P_{0,n}(x).$$
If $\ell$ is odd it becomes instead:
$$2\sum_{j=0}^\ell (-1)^j P_{j,n}(x)=-P_{\ell+1,n+1}(x)+P_{0,n+1}(x).$$
These two identities then imply the statement of part (2).
\end{proof}

\begin{corollary}\label{cor: r-relations}
Let $n\in\mtn$, then:
\begin{enumerate}
\item $\displaystyle \qquad a_1(n+1)=a_0(n).$
\item $\displaystyle \qquad a_0(2n) =(-1)^{n+1}4 a_{n-1}(2n-1) + 4\sum_{j=0}^{n-2}(-1)^ja_j(2n-1) + \begin{cases}-2a_0(2n-1) & n\text{ odd}\\ 0 & n\text{ even}\end{cases}.$
\end{enumerate}
\end{corollary}
\begin{proof}
Part (1) follows from Proposition\,\ref{prop: r-relations}, (2):
$$a_1(n+1) = 2 a_0(n) - a_0(n) = a_0(n).$$
Part (2) follows from noting from Proposition\,\ref{prop: r-relations}, (1) that $a_0(2n)=a_n(2n)+a_n(2n)=2a_n(2n)$. Then apply Proposition\,\ref{prop: r-relations}, (2) again:
$$a_0(2n)=2a_n(2n) = (-1)^{n-1}4\sum_{j=0}^{n-1} (-1)^ja_j(2n-1) +  \begin{cases}-2a_0(2n-1) & n\text{ odd}\\ 2a_0(2n) & n\text{ even}\end{cases}.$$
If $n$ is odd we recover the statement, for $n$ even the statement follows from elementary arithmetic.
\end{proof}

The following intermediate step will be used in order to inductively calculate $\det A(n)$:
\begin{lemma}\label{lemma: A(n)-step1}
Let $n\in\mtn$, then:
$$\det A(n+1) = \det\begin{pmatrix}a_0(n+1)\\ a_0(n)\\ 2a_1(n)\\ 2a_2(n)\\ \vdots \\ 2a_{\lfloor \frac{n}2\rfloor-1}(n)\end{pmatrix},$$
where we treat $a_\ell(n)$ as a row vector of length $\lfloor \frac{n}2\rfloor+1$ instead of $\lfloor \frac{n+1}2\rfloor$. If $n$ is even this means we append a $0$ to its tail.
\end{lemma}
\begin{proof}
Recall:
$$A(n+1)=\begin{pmatrix} a_0(n+1)\\  a_1(n+1)\\ a_2(n+1) \\ \vdots \\ a_{\lfloor \frac{n}2\rfloor}(n+1)\end{pmatrix}$$
First note that $a_1(n+1)=a_0(n)$, and so we can replace the second row of $A(n+1)$ with $a_0(n)$. For the remaining rows we proceed inductively by applying Proposition\,\ref{prop: r-relations}, (2), which for $\ell\geq 1$ we rewrite as:
$$a_{\ell+1}(n+1)=2a_\ell(n)+2\sum_{j=0}^{\ell-1}(-1)^{\ell+j}a_j(n)+\begin{cases}-a_0(n) & \ell\text{ even}\\ a_0(n+1) & \ell\text{ odd}\end{cases},$$
and so we may replace the row $a_{\ell+1}(n+1)$ by $2a_{\ell}(n)$ without changing the determinant.
\end{proof}

\begin{lemma}\label{lemma: A-induc}
Let $n\in\mtn$, then:
$$\det A(2n)=2^n\det A(2n-1),\qquad \det A(2n+1) = (-1)^{n}2^{n-1}\det A(2n).$$
\end{lemma}
\begin{proof}
We begin with $\det A(2n)$. Recall from Corollary\,\ref{cor: r-relations} that
$$a_0(2n) =(-1)^{n+1}4 a_{n-1}(2n-1) + ...$$
where we have suppressed terms involving $a_1(2n-1), ..., a_{n-2}(2n-1)$. It follows that
$$\det A(2n) = \det\begin{pmatrix}(-1)^{n+1} 4a_{n-1}(2n-1)\\ a_0(2n-1)\\ 2a_1(2n-1)\\ 2a_2(2n-1)\\ \vdots \\ 2 a_{n-2}(2n-1)\end{pmatrix} = 2^n \det A(2n-1).$$
In order to calculate $\det A(2n+1)$ recall that in this case that we take $a_\ell(2n)_n=0$ in Lemma\,\ref{lemma: A(n)-step1}. On the other hand $a_0(2n+1)_n=1$, as can be checked directly from the definition. Then:
$$\det A(2n+1) = \det\begin{pmatrix} a_0(2n+1)_{0,...,n-1} & 1 \\ a_0(2n) & 0\\ 2a_1(2n)&0 \\ 2a_2(2n)&0\\ \vdots\\ 2a_{n-1}(2n)&0\end{pmatrix}=(-1)^n 2^{n-1}\det A(2n),$$
which completes the proof of the lemma.
\end{proof}

We then conclude parts (1) and (2) of Theorem\,\ref{combi}:

\begin{corollary}
Let $n\in\mtn$, then
\begin{gather*}
\det A(2n+1) = (-1)^{n (n+1)/2}\, 2^{n^2},\\
\det A(2n) = (-1)^{n(n-1)/2}\, 2^{n(n-1)+1}.
\end{gather*}
\end{corollary}
\begin{proof}
This follows from inductively applying Lemma\,\ref{lemma: A-induc}, and $\det A(1)=1$, $\det A(2)=2$.
\end{proof}

For parts (3) and (4) of Theorem\,\ref{combi} we first note that:
$$\binom{\ell}{2k}\binom{n-\ell}{2(m-k)}=\binom{\ell}{2k} (\binom{n+1-\ell}{2(m-k)+1}-\binom{n-\ell}{2(m-k)+1}),$$
and so:
$$b_\ell(n) = a_\ell(n+1)-a_\ell(n).$$

\begin{lemma}\label{lemma: B(n)-step1}
Let $n\in\mtn$, then:
$$\det B(n) = \det \begin{pmatrix}a_0(n+1)- a_0(n) \\ a_0(n)- a_1(n) \\ a_1(n)- a_2(n)\\ \vdots \\ a_{\lfloor \frac {n-2}2\rfloor}(n)- a_{\lfloor \frac n2\rfloor }(n) \end{pmatrix}.$$
\end{lemma}
\begin{proof}
Rewriting the expression of Proposition\,\ref{prop: r-relations}, (2) as a telescoping sum gives:
$$a_{\ell+1}(n+1) = a_\ell(n) +\sum_{j=0}^{\ell-1}(-1)^{j+\ell-1} (a_{\ell-j}(n) - a_{\ell-j-1}(n)) + \begin{cases}0 & \text{$\ell$ even}\\ a_0(n+1)-a_0(n)& \text{$\ell$ odd}\end{cases}.$$
The desired statement then follows immediately from an induction on rows.
\end{proof}

The following lemma finishes the proof of Theorem\,\ref{combi}:
\begin{lemma}Let $n\in\mtn$, then:
$$\det B(2n) = (-1)^n\frac12 \det A(2n), \qquad \det B(2n-1) = \det A(2n-1).$$
\end{lemma}
\begin{proof}
We begin with the case of $B(2n)$. Using Lemma\,\ref{lemma: B(n)-step1}:
$$\det B(2n)=\det \begin{pmatrix}a_0(2n+1)-a_0(2n) \\ a_0(2n) -a_1(2n)\\ (a_0(2n)-a_1(2n) )+ (a_1(2n)-a_2(2n))\\ \vdots\\ \sum_{j=0}^{n-1} (a_j(2n)-a_{j+1}(2n)) \end{pmatrix}=\det  \begin{pmatrix}a_0(2n+1)-a_0(2n) \\ a_0(2n) -a_1(2n)\\ a_0(2n)-a_2(2n)\\ \vdots\\ a_0(2n)-a_n(2n)\end{pmatrix}.$$
From $a_n(2n)=\frac12 a_0(2n)$ we may subtract $a_0(2n)$ from all rows except the last without changing the determinant, giving:
$$\det B(2n) = \det \begin{pmatrix} a_0(2n+1) \\ -a_1(2n) \\ \vdots \\-a_{n-1}(2n) \\ \frac12 a_0(2n)\end{pmatrix} = \frac 12 \det\begin{pmatrix} a_0(2n+1) \\ a_0(2n) \\ \vdots \\ a_{n-2}(2n) \\ a_{n-1}(2n)\end{pmatrix}.$$
As in the calculation for $A(2n+1)$ the last component of $a_0(2n+1)$ is equal to $1$, while it is $0$ for all $a_\ell(2n)$, whence:
$$\det B(2n) = (-1)^n \frac12 \det A(2n).$$
For the calculation of $\det B(2n-1)$ we first rewrite the expression for $a_0(2n)$ from Corollary\,\ref{cor: r-relations} as a sum of successive differences:
$$a_0(2n)= 2\sum_{j=0}^{n-2}(-1)^j ( a_j(2n-1)-a_{j+1}(2n-1))+\begin{cases}2(a_0(2n-1)-a_{n-1}(2n-1)) & \text{$n$ even}\\ 2a_{n-1}(2n-1)& \text{$n$ odd}\end{cases}.$$
Plugging this into $\det B(2n-1)$ we summarize both cases:
$$\det B(2n-1) =\det \begin{pmatrix}(-1)^n( a_0(2n-1)- 2a_{n-1}(2n-1))\\a_0(2n-1) - a_1(2n-1)\\ \vdots \\ a_{n-2}(2n-1) - a_{n-1}(2n-1) \end{pmatrix}.$$
As in the calculation for $B(2n)$ adding $\sum_{j=0}^{\ell-1} (a_j(2n-1) - a_{j+1}(2n-1))$ to the row containing $a_{\ell-1}(2n-1)-a_\ell(2n-1)$ does not change the determinant. One recovers:
\begin{align*}\det B(2n-1) =&(-1)^n\det \begin{pmatrix}a_0(2n-1)- 2a_{n-1}(2n-1) \\a_0(2n-1) - a_1(2n-1)\\ a_0(2n-1)-a_2(2n-1)\\ \vdots \\ a_{0}(2n-1) - a_{n-1}(2n-1) \end{pmatrix}= (-1)^n\det \begin{pmatrix}-a_{n-1}(2n-1) \\a_0(2n-1) - a_1(2n-1)\\ a_0(2n-1)-a_2(2n-1)\\ \vdots \\ a_{0}(2n-1) - a_{n-1}(2n-1) \end{pmatrix}\\
=& (-1)^n\det \begin{pmatrix}-a_{n-1}(2n-1) \\a_0(2n-1) - a_1(2n-1)\\ a_0(2n-1)-a_2(2n-1)\\ \vdots \\ a_{0}(2n-1)  \end{pmatrix} =  (-1)^n\det \begin{pmatrix}-a_{0}(2n-1) \\- a_1(2n-1)\\ -a_2(2n-1)\\ \vdots \\ - a_{n-1}(2n-1) \end{pmatrix}\\
=&\det A(2n-1),
\end{align*}
whence the claim.
\end{proof}

\bibliographystyle{amsplain}
\bibliography{my}
\end{document}